\documentclass[12pt,draftcls,onecolumn]{IEEEtran}
\usepackage[english]{babel}
\usepackage[latin1]{inputenc}
\usepackage{enumerate}
\usepackage{color}
\usepackage[T1]{fontenc}
\usepackage{subfigure}
\usepackage{dsfont}
\usepackage[T1]{fontenc}
\usepackage{amsmath}
\usepackage{amsthm}
\usepackage{amstext}
\usepackage{amssymb}
\usepackage{mathrsfs}
\usepackage{cite}
\usepackage{mathtools}
\usepackage{tikz}
\usetikzlibrary{arrows}

\def\rF{\mathbb{F}}
\def\R{\mathbb{R}}

\def\argmin{\mathop{\rm arg\, min}}

\def\B{{\mathcal B}}

\def\E{{\mathcal E}}

\def\C{{\mathcal C}}
\def\P{{\mathcal P}}
\def\Q{{\mathcal Q}}

\def\sq{{\mathsf q}}

\def\sf{{\mathsf f}}
\def\sX{{\mathsf X}}
\def\sH{{\mathsf H}}
\def\sA{{\mathsf A}}
\def\sE{{\mathsf E}}
\def\sM{{\mathsf M}}

\theoremstyle{remark}
\newtheorem{definition}{Definition}[section]
\newtheorem{theorem}{Theorem}[section]

\newtheorem{proposition}{Proposition}
\newtheorem{lemma}{Lemma}
\theoremstyle{remark}
\newtheorem{remark}{Remark}
\newtheorem{example}{Example}

\allowdisplaybreaks

\IEEEoverridecommandlockouts

\begin{document}
\sloppy
\title{Quantized Stationary Control Policies in
  Markov Decision Processes
\thanks{The authors are with the Department of Mathematics and Statistics,
     Queen's University, Kingston, ON, Canada,
     Email: \{nsaldi,linder,yuksel\}@mast.queensu.ca}
\thanks{This research was supported in part by
  the Natural Sciences and
  Engineering Research Council (NSERC) of Canada.}
\thanks{The material in this paper was presented in part at the
51st Annual Allerton Conference on Communication, Control and Computing,
Monticello, Illinois, Oct.\ 2013.}
}
\author{Naci Saldi, Tam\'{a}s Linder, Serdar Y\"uksel}
\maketitle
\begin{abstract}
  For a large class of Markov Decision Processes, stationary (possibly
  randomized) policies are globally optimal. However, in Borel state and action
  spaces, the computation and implementation of even such stationary policies
  are known to be prohibitive. In addition, networked control applications
  require remote controllers to transmit action commands to an actuator with low
  information rate. These two problems motivate the study of approximating
  optimal  policies by quantized (discretized)
  policies. To this end, we introduce deterministic stationary quantizer
  policies and show that such policies can approximate optimal deterministic
  stationary policies with arbitrary precision under mild technical conditions,
  thus demonstrating that one can search for $\varepsilon$-optimal policies
  within the class of quantized control policies.  We also derive explicit
  bounds on the approximation error in terms of the rate of the approximating
  quantizers.  We extend all these approximation results to randomized
  policies.  These findings pave the way toward applications in optimal design of
  networked control systems where controller actions need to be quantized, as
  well as for new computational methods for generating approximately optimal
  decision policies in general (Polish) state and action spaces for both
  discounted cost and average cost.
\end{abstract}

\begin{keywords}
Markov decision processes, stochastic control, approximation, quantization,
stationary policies.
\end{keywords}

\section{Introduction}
\label{sec0}

In the theory of Markov decision processes (MDPs), control policies
induced by measurable mappings from state to the action space are called stationary. For a large class of infinite horizon optimization problems, the set of stationary policies is the smallest structured set of control policies in
which one can find a globally optimal policy. However, computing an optimal policy even in this
class is in general computationally prohibitive for non-finite Polish (that is,
complete and separable metric) state and action spaces. Furthermore, in
applications to networked control, the transmission of such control actions to an
actuator is not realistic when there is an information transmission constraint
(imposed by the presence of a communication channel) between a plant,
a controller, or an actuator.

Hence, it is of interest to study the approximation of optimal stationary policies. Several approaches have been developed in the literature to tackle this problem, most of which assume finite or countable state spaces, see \cite{BeTs96,ReKr02,Ort07,Whi82,Cav86,DuPr12}. In this paper, we study the following question: for infinite Borel state and action spaces, how much is lost in performance if optimal policy is represented with a finite number of bits?
This formulation appears to be new in the networked control literature, where stability properties of quantized control actions have been studied extensively, but the optimization of quantized control actions has not been studied as much in the context of cost minimization.

This paper contains two main contributions: (i) We establish conditions under which quantized control policies are asymptotically optimal; that is, as the accuracy of quantization increases, the optimal cost is achieved as the limit of the cost of quantized policies. (ii) We establish rates of convergence  under further conditions; that is, we obtain bounds on the approximation loss due to quantization. These findings are somewhat analogous to results in optimal quantization theory \cite{GrNe98}.

The rest of the paper is organized as follows. In Section \ref{sec1} we review
the definition of discrete time Markov decision processes (MDP) in the setting
we will be dealing with. In Section \ref{sub1sec2} we consider the approximation
problem for the total and discounted cost cases  using strategic measures (that is,
measures on the sequence space of states and control actions). In Section \ref{sub2sec2}
a similar approximation result is
obtained for the average cost case using ergodic invariant probability measures of
the induced Markov chains. In Section \ref{sec3} we derive quantitative bounds
on the approximation error in terms of the rate of the approximating quantizers
for both discounted and average costs. In Section \ref{sec4} we
extend the results of Sections \ref{sec2} and \ref{sec3} to approximating
randomized stationary policies  by randomized stationary
quantizer policies. Finally, in Section~\ref{sec5} we discuss future research directions.

\section{Markov Decision Processes}\label{sec1}

For a metric space $\sE$, let $\B(\sE)$ denote its Borel $\sigma$-algebra. Unless otherwise specified, the term "measurable" will refer to Borel measurability. We denote by $\P(\sE)$ the set of all probability measures on $\sE$.

Consider a discrete time Markov decision process (MDP) with \emph{state space} $\sX$ and \emph{action space} $\sA$, where $\sX$ and $\sA$ are complete, separable metric (Polish) spaces equipped with their Borel $\sigma$-algebras $\B(\sX)$ and $\B(\sA)$, respectively. For all $x \in \sX$, we assume that the \emph{set of admissible actions} is $\sA$. Let the \emph{stochastic kernel} $p(\,\cdot\,|x,a)$ denote the \emph{transition probability} of the next state given that previous state-action pair is $(x,a)$ \cite{HeLa96}. The probability measure $\mu$ over $\sX$ denotes the initial distribution.

Define the history spaces
$\sH_{n}=(\sX\times\sA)^{n}\times\sX$, $n=0,1,2,\ldots$ endowed with their
product Borel $\sigma$-algebras generated by $\B(\sX)$ and $\B(\sA)$. A
\emph{policy} is a sequence $\pi=\{\pi_{n}\}_{n\geq0}$ of stochastic kernels
on $\sA$ given $\sH_{n}$. A policy $\pi$ is said to be \emph{deterministic} if the stochastic kernels $\pi_n$
are realized by a sequence of measurable functions $\{f_{n}\}$ from $\sH_{n}$ to $\sA$, i.e.,
$\pi_{n}(\,\cdot\,|h_{n})=\delta_{f_{n}(h_{n})}(\,\cdot\,)$ where
$f_{n}:\sH_{n}\rightarrow\sA$ is measurable. A policy $\pi$ is called \emph{stationary} if the stochastic kernels $\pi_{n}$  depend only on the current state; that is,
$\pi_{n}=\pi_{m}$ $(m,n\geq0)$ and $\pi_n$ is a stochastic kernel on $\sA$ given $\sX$. A policy $\pi$ that is both
deterministic and stationary is called \emph{deterministic stationary}. Hence, deterministic stationary policies are defined
by a measurable function $f:\sX\to \sA$. We denote by $S$ the set of deterministic stationary policies.

According to the Ionescu Tulcea theorem \cite{HeLa96}, an initial distribution $\mu$ on $\sX$ and a policy $\pi$ define a unique probability measure $P_{\mu}^{\pi}$ on $\sH_{\infty}=(\sX\times\sA)^{\infty}$, which is called a \emph{strategic measure} \cite{Fei96}. Thus $P_{\mu}^{\pi}$ is symbolically given by
\begin{align}
P_{\mu}^{\pi}(dx_0 da_0 dx_1 da_1 \ldots)\coloneqq \prod_{n=0}^{\infty} p(dx_n|x_{n-1},a_{n-1})\pi(da_n|h_n),\nonumber
\end{align}
where $h_n=(x_0,a_0,\ldots,x_{n-1},a_{n-1},x_n)$ and $p(dx_0|x_{-1},a_{-1})=\mu(dx_0)$. The expectation with respect to $P_{\mu}^{\pi}$ is denoted by $E_{\mu}^{\pi}$. If $\mu=\delta_x$ for some $x\in\sX$, we write $P_{x}^{\pi}$ and $E_x^{\pi}$ instead of $P_{\delta_x}^{\pi}$ and $E_{\delta_x}^{\pi}$, respectively. Hence, given any policy $\pi$ and an initial distribution $\mu$, $\{x_n,a_n\}_{n\geq1}$ is a $\sX \times \sA$-valued stochastic process defined on a probability space $\bigl(\sH_{\infty},\B(\sH_{\infty}),P_{\mu}^{\pi}\bigr)$ satisfying $P_{\mu}^{\pi}(x_0 \in \,\cdot\,) = \mu(\,\cdot\,)$,
$P_{\mu}^{\pi}(x_n \in \,\cdot\, | h_{n-1}, a_{n-1}) = P_{\mu}^{\pi}(x_n \in \,\cdot\, | x_{n-1},a_{n-1}) = p(\,\cdot\, | x_{n-1},a_{n-1})$, and $P_{\mu}^{\pi}(a_n \in \,\cdot\, | h_n) = \pi_n(\,\cdot\, | h_n)$, for all $n$.

Let $c$ and $c_{n}$, $n=0,1,2,\ldots$, be measurable functions from $\sX\times\sA$ to $[0,\infty)$. The cost functions $w$ considered in this paper are the following.
\begin{itemize}
\item[i)]Expected Total Cost: $w_t(\pi,\mu) \coloneqq E_{\mu}^{\pi}\bigl[\sum_{n=0}^{\infty}c_{n}(x_{n},a_{n})\bigr]$.
\item[ii)]Expected Discounted Cost: $w_{\beta}(\pi,\mu) \coloneqq E_{\mu}^{\pi}\bigl[\sum_{n=0}^{\infty}\beta^{n}c(x_{n},a_{n})\bigr]$ for some $\beta\in(0,1)$.
\item[iii)]Expected Average Cost: $w_{A}(\pi,\mu) \coloneqq \limsup_{N\rightarrow\infty}\frac{1}{N}E_{\mu}^{\pi}\bigl[\sum_{n=0}^{N}c(x_{n},a_{n})\bigr]$.
\end{itemize}
Note that the expected discounted cost is a special case of the expected total cost. Define $L_{\Delta,\mu}\coloneqq\{P_{\mu}^{\pi}: \pi\in\Delta\}$. Then $L_{\Delta,\mu}$ is the set of all strategic measures with the initial distribution $\mu$. Hence, the cost function $w$ can be viewed as a function from $L_{\Delta,\mu}$ to $[0,\infty]$.

We write $w(\pi,\mu)$ to denote the cost function (either i), ii), or iii)) of the policy $\pi$ for the initial distribution $\mu$. If $\mu = \delta_x$, we write $w(\pi,x)$ instead of $w(\pi,\delta_x)$. A policy $\pi^{*}$ is called optimal if $w(\pi^{*},\mu) = \inf_{\pi \in \Delta} w(\pi,\mu)$ for all $\mu \in \P(\sX)$. It is well known that the set of deterministic stationary policies is optimal for a large class of infinite horizon discounted cost problems (see, e.g., \cite{HeLa96,FeKaZa12}) and average cost optimal control problems (see, e.g., \cite{Bor02,FeKaZa12}). For instance, Feinberg  \emph{et al.\ }\cite{FeKaZa12} (see also \cite{FeKaZa13}) recently showed the existence of an optimal stationary policy for discounted cost under weak continuity of the transition probability and $\mathbb{K}$-inf-compactness of the one-stage cost function, and for average cost with an additional mild assumption.

Throughout the paper, the initial distribution $\mu$ is assumed to be an arbitrary fixed distribution unless otherwise is specified.

\subsection{Notation and Conventions}\label{sub0sec1}

The set of all bounded measurable real functions and bounded continuous real functions on a metric space $\sE$  are denoted by $B(\sE)$ and $C_{b}(\sE)$, respectively.  For any $\nu \in \P(\sE)$ and measurable real function $g$ on $\sE$, define $\nu(g) \coloneqq \int g d\nu$.
Let $\sE_n=\prod_{i=1}^{n} \sE_i$ $(2\leq n\leq \infty)$ be finite or a infinite
product space. By an abuse of notation, any function $g$ on $\prod_{j=i_1}^{i_m}
\sE_j$, where $\{i_1,\ldots,i_m\}\subseteq\{1,\ldots,n\}$ ($m\leq n$), is also treated as a
function on $\sE_n$ by identifying it with its natural extension to $\sE_n$.
For any $\pi$ and initial distribution $\mu$, let $\lambda^{\pi,\mu}_n$, $\lambda^{\pi,\mu}_{(n)}$ and $\gamma^{\pi,\mu}_n$, respectively, denote the law of $x_n$, $(x_0,\ldots,x_n)$ and $(x_n,a_n)$ for all $n\geq0$. Hence, for instance, we may write $\lambda_{(n+1)}^{\pi,\mu}(h) = \lambda_{(n)}^{\pi,\mu}\bigl(\lambda_{(1)}^{\pi,x_n}(h)\bigr)$ where $h \in B(\sX^{n+2})$.
Let $\rF$ denote the set of all measurable functions from $\sX$ to $\sA$. For any $g \in B(\sH_n)$ ($n\geq1$) and $f \in \rF$, define $g_{f}(x_0,\ldots,x_n) \coloneqq g(x_0,f(x_0),\ldots,f(x_{n-1}),x_n)$.
Hence, when $c \in B(\sX \times \sA)$, $c_f(x_n) = c(x_n,f(x_n))$ since $c \in B(\sH_{n+1})$ by our conventions.

\subsection{Problem Formulation}\label{sub1sec1}

In this section we give a formal definition of the problems considered in this paper. To this end, we first give the definition of a quantizer.

\begin{definition}\label{def1}
A measurable function $q: \sX\rightarrow\sA$ is called a \emph{quantizer} from
$\sX$ to $\sA$ if the range of $q$, i.e., $q(\sX)=\{q(x)\in\sA:x\in\sX\}$, is
finite.
\end{definition}

The elements of $q(\sX)$ (i.e., the possible values of $q$) are called
the \emph{levels} of $q$.
 The rate $R$ of a quantizer $q$ is defined as the
logarithm of the number of its levels: $R=\log_2|q(\sX)|$.
Note that $R$
(approximately) represents the number of bits needed to losslessly encode the
output levels of $q$ using binary codewords of equal length.
Let $\Q$ denote the set
of all quantizers from $\sX$ to $\sA$. In this paper we introduce a new type of
policy called a \emph{deterministic stationary quantizer policy}.  Such a policy
is a constant sequence $\pi=\{\pi_{n}\}$ of stochastic kernels on $\sA$ given
$\sX$ such that $\pi_{n}(\,\cdot\,|x)=\delta_{q(x)}(\,\cdot\,)$ for all $n$ for
some $q\in\Q$. For any finite set $\Lambda \subset \sA$, let $\Q(\Lambda)$ denote the set
of all quantizers having range $\Lambda$ and let $S\Q(\Lambda)$ denote the set of all deterministic stationary
quantizer policies induced by $\Q(\Lambda)$.

The principal goal in this paper is to determine conditions on the spaces $\sX$ and $\sA$, initial distribution $\mu$, the stochastic kernel
$p$, and the one-stage cost functions $c$, $c_n$ ($n\geq0$) such that there exists a sequence of finite subsets $\{\Lambda\}_{k\geq1}$ of $\sA$ for which the following statements hold:
\begin{itemize}
\item [\textbf{(P1)}] For any $\pi \in S$ there exists an approximating sequence $\{\pi^k\}$ satisfying $\lim_{k\rightarrow\infty} w(\pi^{k},\mu) = w(\pi,\mu)$, where $\pi^{k}\in S\Q(\Lambda_k)$ ($k\geq1$).
\item [\textbf{(P2)}] For any $\pi \in S$ the approximating sequence $\{\pi^{k}\}$ in \textbf{(P1)} is such that $|w(\pi,\mu)-w(\pi^{k},\mu)|$ can be explicitly upper bounded by a term depending on the cardinality of $\Lambda_k$.
\end{itemize}

\noindent Thus \textbf{(P1)} implies the existence of a sequence of stationary quantizer policies converging to an optimal stationary policy, while \textbf{(P2)} implies that the approximation error can be explicitly controlled.

\section{Approximation of deterministic stationary policies}
\label{sec2}

A sequence $\{\mu_n\}$ of measures on a  measurable space $(\sE,\E)$ is said to
converge setwise \cite{HeLa03} to a measure $\mu$ if $\mu_n(B)\rightarrow\mu(B)
\text{ for all } B\in\E$, or equivalently, $\mu_n(g) \rightarrow \mu(g)$
for all $g\in B(\sE)$. In this section, we will impose the following
assumptions:
\begin{itemize}
\item [(a)] The stochastic kernel $p(\,\cdot\,|x,a)$ is setwise continuous in $a\in\sA$, i.e., if $a_{n}\rightarrow a$, then $p(\,\cdot\,|x,a_{n})\rightarrow p(\,\cdot\,|x,a)$ setwise for all $x\in\sX$.
\item [(b)] $\sA$ is compact.
\end{itemize}

\begin{remark}\label{remark0}

 Note that if $\sX$ is countable, then $B(\sX) = C_b(\sX)$ which implies the equivalence of setwise convergence and weak convergence. Hence, results developed in this paper are applicable to the MDPs having weakly continuous, in the action variable, transition probabilities when the state space is countable.

\end{remark}

We now define the $ws^{\infty}$ topology on $\P(\sH_{\infty})$ which was first
introduced by Sch\"{a}l in \cite{Sch75}. Let $\C(\sH_{0})=B(\sX)$ and let
$\C(\sH_{n})$ ($n\geq1$) be the set of real valued functions $g$ on $\sH_{n}$
such that $g\in B(\sH_{n})$ and
$g(x_{0},\,\cdot\,,x_{1},\,\cdot\,,\ldots,x_{n-1},\,\cdot\,,x_{n})\in
C_{b}(\sA^{n})$ for all $(x_0,\ldots,x_n)\in \sX^{n+1}$. The $ws^{\infty}$
topology on $\P(\sH_{\infty})$ is defined as the smallest topology which renders
all mappings $P\mapsto P(g)$,
$g\in\bigcup_{n=0}^{\infty}\C(\sH_{n})$, continuous. Similarly, the weak
topology on $\P(\sH_{\infty})$ can also be defined as the smallest topology
which makes all  mappings $P\mapsto P(g)$,
$g\in\bigcup_{n=0}^{\infty}C_b(\sH_{n})$, continuous \cite[Lemma 4.1]{Sch75}. A
theorem due to  Balder \cite[page 149]{Bal89} and  Nowak \cite{Now88}
states that the weak topology and the $ws^{\infty}$ topology on $L_{\Delta,\mu}$ are
equivalent under the assumptions (a) and (b). Hence, the $ws^{\infty}$ topology is metrizable with the Prokhorov
metric on $L_{\Delta,\mu}$.

The following theorem is a Corollary of \cite[Theorem 2.4]{Ser82} which will be used in this paper frequently.
It is a generalization of the dominated convergence theorem.

\begin{theorem}
  Let $(\sE,\E)$ be a measurable space and let $\nu$, $\nu_{n}$ $(n\geq1)$ be
  measures with the same finite total mass. Suppose $\nu_{n}\rightarrow\nu$
  setwise, $\lim_{n\rightarrow\infty}h_{n}(x)=h(x)$ for all $x\in\sX$, and $h$,
  $h_{n}$ ($n\geq1$) are uniformly bounded. Then, $\lim_{n\rightarrow\infty} \nu_{n}(h_n)=\nu(h)$.
\label{cor1}
\end{theorem}

Let $d_{\sA}$ denote the metric on $\sA$. Since the action space $\sA$ is compact and thus totally bounded, one can find a sequence of finite sets $\bigl(\{a_i\}_{i=1}^{m_k}\bigr)_{k\geq1}$ such that for all $k$,
\begin{align}
\min_{i\in\{1,\ldots,m_k\}} d_{\sA}(a,a_i) < 1/k \text{ for all } a \in \sA. \label{eq22}
\end{align}
In other words, $\{a_i\}_{i=1}^{m_k}$ is an $1/k$-net in $\sA$. Let $\Lambda_k \coloneqq \{a_1,\ldots,a_{m_k}\}$ and for any $f \in \rF$ define the sequence $\{q_{k}\}$ by letting
\begin{align}
q_k(x) \coloneqq \argmin_{a \in \Lambda_k} d_{\sA}(f(x),a), \label{eq23}
\end{align}
where ties are broken so that $q_k$ are measurable. Note that, $q_k \in \Q(\Lambda_k)$ for all $k$ and $q_{k}$ converges \emph{uniformly} to $f$ as $k\to \infty$. Let $\pi \in S$ and $\pi^k \in S\Q(\Lambda_k)$ be induced by $f$ and $q_k$, respectively. We call each $\pi_k$ a \emph{quantized approximation} of $\pi$. In the rest of this paper, we assume that the sequence $\{\Lambda_k\}$, as defined above, is fixed.

\begin{remark} \label{remark2}
Since $\sA$ is separable, there exists a totally bounded metric $\tilde{d}_{\sA}$ on $\sA$ that is compatible with the original metric structure of $\sA$ \cite[Corollary 3.41]{ChBo06}. Hence, compact action space $\sA$ is indeed not necessary for the problem \textbf{(P1)}. However, it is usually necessary to show the existence of an optimal deterministic stationary policy.
\end{remark}


\subsection{Expected Total and Discounted Costs}
\label{sub1sec2}

Here we consider the first approximation problem $\textbf{(P1)}$ for
the expected total cost criterion and its special case, the
expected discounted cost criterion (see Section~\ref{sec1}). Recall that $w_t$ and $w_{\beta}$ denote the expected total and discounted costs, respectively. We impose the following assumptions in addition to assumptions (a) and (b):
\begin{itemize}
\item [(c)] $c$ and $c_{n}$ ($n\geq1$) are non-negative, bounded measurable functions satisfying $c(x,\,\cdot\,)$, $c_{n}(x,\,\cdot\,)\in C_{b}(\sA)$ for all $x\in\sX$.
\item [(d)] $\sup_{\tilde{\pi}\in S}\sum_{n=N+1}^{\infty} \gamma^{\tilde{\pi},\mu}_n(c_n) \rightarrow 0$ as $N\rightarrow\infty$.
\end{itemize}

\begin{remark}\label{remark3}

We note that all the results in this paper remain valid if it is only assumed that $c$ and $c_n$ ($n\geq0$) are bounded and measurable.
\end{remark}

Since the one-stage cost functions $c_{n}$ are non-negative, assumption (d) is
equivalent to Condition (C) in \cite[pg. 349]{Sch75}. Clearly, the expected
discounted cost satisfies assumption (d) under assumption (c). We now state our main theorem in this subsection.

\begin{theorem}
Suppose assumptions (a), (b), (c) hold. Let $\pi \in S$ and $\{\pi^k\}$ be the quantized approximations of $\pi$. Then, $w_{\beta}(\pi^k,\mu) \rightarrow w_{\beta}(\pi,\mu)$ as $k\rightarrow\infty$. The same statement is true for $w_t$ if we further impose assumption (d).
\label{thm3}
\end{theorem}

The proof of Theorem~\ref{thm3} requires the following proposition which is proved in Appendix~\ref{app1}.

\begin{proposition}
Suppose assumptions (a) and (b) hold. Then for any $\pi \in S$, the strategic measures $\{P_{\mu}^{\pi^k}\}$ induced by the quantized approximations $\{\pi^k\}$ of $\pi$ converge to the strategic measure $P_{\mu}^{\pi}$ of $\pi$ in the $ws^{\infty}$ topology. Hence, $\gamma^{\pi^k,\mu}_n(c_n) \rightarrow \gamma^{\pi,\mu}_n(c_n)$ as $k\rightarrow\infty$ under assumption (c).
\label{prop1}
\end{proposition}

\begin{proof}[Proof of Theorem~\ref{thm3}]
Since $w_{\beta}$ is a special case of $w_t$ and satisfies (d) under assumption (c), it is enough to prove the theorem for $w_t$. By Proposition~\ref{prop1},  $\gamma^{\pi^k,\mu}_n(c_n) \rightarrow \gamma^{\pi,\mu}_n(c_n)$ as $k\rightarrow\infty$ for all $n$. Then, we have
\begin{align}
\limsup_{k\rightarrow\infty}|w_t(\pi^k,\mu) - w_t(\pi,\mu)|
&\leq\limsup_{k\rightarrow\infty}\sum_{n=0}^{\infty} |\gamma^{\pi^k,\mu}_n(c_n) - \gamma^{\pi,\mu}_n(c_n)| \nonumber \\
&\leq \lim_{k\rightarrow\infty}\sum_{n=0}^{N} |\gamma^{\pi^k,\mu}_n(c_n) - \gamma^{\pi,\mu}_n(c_n)|
+2\sup_{\tilde{\pi}\in S}\sum_{n=N+1}^{\infty}\gamma^{\tilde{\pi},\mu}_n(c_n) \nonumber\\
&=2\sup_{\tilde{\pi}\in S}\sum_{n=N+1}^{\infty} \gamma^{\tilde{\pi},\mu}_n(c_n).\nonumber
\end{align}
Since the last expression converges to zero as $N\rightarrow\infty$ by assumption (d), the proof is complete.
\end{proof}

\begin{remark}\label{remark1}
Notice that this proof implicitly shows that $w_t$ and $w_{\beta}$ are sequentially continuous with respect to the strategic measures in the $ws^{\infty}$ topology.
\end{remark}

The following is a generic example frequently considered in the theory of Markov decision processes (see \cite[p. 496]{HeLa94}, \cite{HeRo01}, \cite[p. 23]{HeLa99}).

\begin{example}
\label{exm1}
Let us consider an additive-noise system given by
\begin{align}
x_{n+1}=F(x_{n},a_{n})+v_{n}, \text{ } n=0,1,2,\ldots \nonumber
\end{align}
where $\sX=\R^{n}$ and the $v_{n}$'s are independent and identically distributed (i.i.d.) random vectors whose common
distribution has a continuous, bounded, and
strictly positive probability density function. A
non-degenerate Gaussian distribution satisfies this condition. We assume that
the action space $\sA$ is a compact subset of $\R^{d}$ for some $d\geq1$, the one stage cost
functions $c$ and $c_n$ $(n\geq1)$ satisfy assumption (c),  and $F(x,\,\cdot\,)$
is continuous for all $x\in\sX$. It is straightforward to show that
assumption (a) holds under these conditions. Hence, under assumption (d) on the
cost functions $c_n$, Theorem
\ref{thm3} holds for this system.
\end{example}

\subsection{Expected Average Cost}
\label{sub2sec2}

In this section we consider the first approximation problem $\textbf{(P1)}$ for
the expected average cost criterion (see Section~\ref{sec1}). We are still assuming (a), (b), and (c). In contrast to
the expected total and discounted cost criteria, the expected average cost is
in general not sequentially continuous with respect to strategic measures for the $ws^{\infty}$ topology
under practical assumptions. Instead, we develop an approach based on the convergence of the sequence of invariant probability measures under quantized stationary policies.

Recall that $w_A$ denotes the expected average cost. Observe
that any deterministic stationary policy $\pi$, induced by $f$, defines a
stochastic kernel on $\sX$ given $\sX$ via
\begin{align}
Q_{\pi}(\,\cdot\,|x) \coloneqq \lambda^{\pi,x}_1(\,\cdot\,) = p(\,\cdot\,|x,f(x)).
\label{eq4}
\end{align}
Let us write $Q_{\pi}g(x) \coloneqq \lambda^{\pi,x}_1(g)$. If $Q_{\pi}$ admits an ergodic invariant probability measure $\nu_{\pi}$, then by Theorem
2.3.4 and Proposition 2.4.2 in \cite{HeLa03}, there exists an invariant set with
full $\nu_{\pi}$ measure such that for all $x$ in that set we have
\begin{align}
w_A(\pi,x)&=\limsup_{N\rightarrow\infty}\frac{1}{N} \sum_{n=0}^{N-1} \gamma^{\pi,\mu}_n(c) \nonumber \\
&=\lim_{N\rightarrow\infty}\frac{1}{N} \sum_{n=0}^{N-1} \lambda^{\pi,x}_n(c_{f}) = \nu_{\pi}(c_{f}).
\label{neweq15}
\end{align}
Let $\sM_{\pi}\in\B(\sX)$ be the set of all $x\in\sX$ such that convergence in
(\ref{neweq15}) holds. Hence, $\nu_{\pi}(\sM_{\pi})=1$ if $\nu_{\pi}$ exists. The following assumptions will be imposed in the main theorem of this section.
\begin{itemize}
\item [(e)] For any $\pi\in S$, $Q_{\pi}$ has a unique invariant probability measure $\nu_{\pi}$.
\item [(f1)] The set $\Gamma_{S}\coloneqq\{\nu\in\P(\sX): \nu Q_{\pi}=\nu \text{ for some } \pi\in S\}$ is relatively sequentially compact in the setwise topology.
\item [(f2)] There exists $x\in\sX$ such that for all $B\in\B(\sX)$, $\lambda^{\pi,x}_{n}(B)\rightarrow\nu_{\pi}(B)$ uniformly in $\pi\in S$.
\item [(g)] $\sM\coloneqq\bigcap_{\pi\in S}\sM_{\pi}\neq \emptyset$.
\end{itemize}

\begin{theorem}
Let the initial distribution $\mu$ be concentrated on some $x\in\sM$. Let $\pi \in S$ and $\{\pi^k\}$ be the quantized
approximations of $\pi$. Then, $w_A(\pi^k,\mu) \rightarrow w_A(\pi,\mu)$ under the assumptions (e), (f1) or (f2), and (g).
\label{thm4}
\end{theorem}

\begin{proof}
See Appendix~\ref{app2}.
\end{proof}

In the rest of this section we will derive conditions under which assumptions (e), (f1), (f2), and (g) hold.
To begin with, assumptions (e), (f2) and (g) are satisfied under any of the
conditions $Ri$, $i\in\{0,1,1(a),1(b),2,\ldots,6\}$ in
\cite{HeMoRo91}. Moreover, $\sM=\sX$ in (g) if at least one of the above
conditions holds. The next step is to find sufficient conditions for assumptions
(e), (f1) and (g) to hold.

Observe that the stochastic kernel $p$ on $\sX$ given $\sX\times\sA$ can be written as a measurable mapping from $\sX\times\sA$ to $\P(\sX)$ if $\P(\sX)$ is equipped with its Borel $\sigma$-algebra generated by the weak topology \cite{Bil99}, i.e.,
$p(\,\cdot\,|x,a):\sX\times\sA\rightarrow\P(\sX)$.
We impose the following assumption:
\begin{itemize}
\item [(e1)] $p(\,\cdot\,|x,a)\leq\zeta(\,\cdot\,)$ for all $x\in\sX$, $a\in\sA$ for some finite measure $\zeta$ on $\sX$.
\end{itemize}

\begin{proposition}
Suppose (e1) holds. Then, for any $\pi \in S$ induced by $f$, $Q_{\pi}$ has an invariant probability measure $\nu_{\pi}$. Furthermore, $\Gamma_S$ is sequentially relatively compact in the setwise topology. Hence, (e1) implies assumption (f1). In addition, if these invariant measures are unique, then assumptions (e) and (g) also hold with $\sM=\sX$ in (g).
\label{fact3}
\end{proposition}

\begin{proof}
For any $\pi \in S$, define $Q^{(N)}_{\pi,x}(\,\cdot\,)\coloneqq\frac{1}{N}\sum_{n=0}^{N-1}\lambda^{\pi,x}_{n}(\,\cdot\,)$ for some $x \in \sX$. Clearly, $Q^{(N)}_{\pi,x}\leq\zeta$ for all $N$. Hence, by \cite[Corollary 1.4.5]{HeLa03} there exists a subsequence $\{Q^{(N_{k})}_{\pi,x}\}$ which converges to some probability measure $\nu_{\pi}$ setwise. Following the same steps in \cite[Theorem 4.17]{Hai06} one can show that $\nu_{\pi}(g) = \nu_{\pi}(Q_{\pi}g)$,
for all $g \in B(\sX)$. Hence, $\nu_{\pi}$ is an invariant probability measure for $Q_{\pi}$.

Furthermore, assumption (e1) implies $\nu_{\pi} \leq \zeta$ for all $\nu_{\pi} \in \Gamma_s$. Thus, $\Gamma_s$ is relatively sequentially compact in the setwise topology  by again \cite[Corollary 1.4.5]{HeLa03}.

Finally, for any $\pi$, if the invariant measure $\nu_{\pi}$ is unique, then every setwise convergent subsequence of the relatively sequentially compact sequence $\{Q_{\pi,x}^{(N)}\}$ must converge to $\nu_{\pi}$. Hence, $Q_{\pi,x}^{(N)}\rightarrow\nu_{\pi}$ setwise which implies that $w_A(\pi,x)=\limsup_{N\rightarrow\infty} Q_{\pi,x}^{(N)}(c_{f})=\lim_{N\rightarrow\infty} Q_{\pi,x}^{(N)}(c_{f})=\nu_{\pi}(c_{f})$ for all $x\in\sX$ since $c_{f} \in B(\sX)$. Thus, $\sM=\sX$ in (g).
\end{proof}

\begin{example}
  Let us consider an additive-noise system in Example \ref{exm1} with the same assumptions.
  Furthermore, we assume $F$ is bounded. Observe that for any $\pi\in S$, if $Q_{\pi}$ has an
  invariant probability measure, then it has to be unique \cite[Lemma 2.2.3]{HeLa03}
  since there cannot exist disjoint invariant sets due to the
  positivity of $g$. Since this system satisfies (e1) and $R1(a)$
  in \cite{HeMoRo91} due to the boundedness of $F$, assumptions (e), (f1), (f2) and (g) hold with $\sM=\sX$.
  This means that Theorem~\ref{thm4} holds for an additive noise system under the above conditions.
\end{example}

\section{Rates of Convergence}
\label{sec3}

In this section we consider the problem \textbf{(P2)} for the discounted and average cost criteria.
Let $\|\,\cdot\,\|_{TV}$ \cite{HeLa03} denote the total variation distance between measures. We will impose a new set of assumptions in this section:
\begin{itemize}
\item [(h)] $\sA$ is infinite compact subset of $\R^{d}$ for some $d\geq1$.
\item [(j)] $c$ is bounded and $|c(x,\tilde{a})-c(x,a)|\leq K_{1} d_{\sA}(\tilde{a},a)$ for all $x$, and some $K_1\geq0$.
\item [(k)] $\|p(\,\cdot\,|x,\tilde{a})-p(\,\cdot\,|x,a)\|_{TV}\leq K_{2}d_{\sA}(\tilde{a},a)$ for all $x$, and some $K_2\geq0$.
\item [(l)] There exists positive constants $C$ and $\beta \in (0,1)$ such that for all $\pi \in S$, there is a (necessarily unique) probability measure $\nu_{\pi} \in \P(\sX)$ satisfying
    $\| \lambda^{\pi,x}_n - \nu_{\pi} \|_{TV} \leq C \kappa^n \text{  for all } x \in \sX \text{ and } n\geq1.$
\end{itemize}
Assumption (l) implies that for any policy $\pi \in S$, the stochastic kernel $Q_{\pi}$, defined in (\ref{eq4}), has a unique invariant probability measure $\nu_{\pi}$ and satisfies \emph{geometric ergodicity} \cite{HeLa99}. Note that (l) holds under any of the conditions $Ri$, $i\in\{0,1,1(a),1(b),2,\ldots,5\}$ in \cite{HeMoRo91}. Moreover, one can explicitly compute the constants $C$ and $\kappa$ for certain systems. For instance, consider an additive-noise system in Example \ref{exm1} with
Gaussian noise. Let $\sX=\R$. Assume $F$ has a bounded range so that
$F(\R)\subset [-L, L]$  for some $L>0$. Let $m$ denote the Lebesgue measure
on $\R$. Then, assumption (l) holds with $C=2$ and $\kappa=1-\varepsilon L$, where $\varepsilon=\frac{1}{\sigma\sqrt{2\pi}}\exp^{-(2L)^{2}/2\sigma^{2}}$.
For further conditions that imply (l) we refer the reader to \cite{HeMoRo91}, \cite{HeLa99}, \cite{MeTw94}.

Assumptions (h), (j) and (k) will be imposed for both cases, but (l) will only be assumed for the expected average cost. The following example gives the sufficient conditions for the additive noise system under which (j), (k) and (l) hold.

\begin{example}
  Consider the additive-noise system in Example \ref{exm1}. In addition
to the assumptions there, suppose $F(x,\,\cdot\,)$ is  Lipschitz
  uniformly in  $x\in\sX$ and the common density $g$ of the $v_n$ is Lipschitz on all
  compact subsets of $\sX$. Note that a Gaussian density has these properties. Let
  $c(x,a)\coloneqq \|x-a\|^2$.  Under these conditions,
  assumptions (j) and (k) hold for the additive noise system. If we further assume that $F$ is
  bounded, then assumption (l) holds as well.
\end{example}

The following result is a consequence of the fact that if $\sA$ is a
compact subset of $\R^d$ then there exist a constant $\alpha>0$ and finite
subsets $\Lambda_k\subset\sA$ with cardinality $|\Lambda_k| = k$ such that
$\max_{x\in\sA}\min_{y\in \Lambda_k} d_{\sA}(x,y)\leq \alpha (1/k)^{1/d}$ for all $k$,
where $d_{\sA}$ is the Euclidean distance on $\sA$ inherited from $\R^{d}$.

\begin{lemma}
  Let $\sA\subset\R^{d}$ be compact. Then for any measurable function
  $f:\sX\rightarrow\sA$ we can construct a sequence of quantizers $\{q_{k}\}$
  from $\sX$ to $\sA$ which
  satisfy  $\sup_{x\in\sX}
  d_{\sA}(q_{k}(x),f(x))\leq \alpha (1/k)^{1/d}$ for some constant $\alpha$.
\label{fact4}
\end{lemma}

The following proposition is the key result in this section. It is proved in Appendix~\ref{app3}

\begin{proposition}
Let $\pi\in S$ and $\{\pi^{k}\}$ be the quantized approximations of $\pi$. For any initial distribution $\mu$ we have
\begin{align}
\|\lambda^{\pi,\mu}_n - \lambda^{\pi^k,\mu}_n\|_{TV}\leq \alpha K_{2}(2n-1)(1/k)^{1/d}
\label{eq6}
\end{align}
for all $n\geq1$ under assumptions (h), (j), and (k).
\label{prop8}
\end{proposition}

\subsection{Expected Discounted Cost}
\label{sub1sec3}

The proof of the following theorem essentially follows from Proposition~\ref{prop8}. The proof is given in Appendix~\ref{app4}.

\begin{theorem}
  Let $\pi\in S$ and $\{\pi^{k}\}$ be the quantized approximations of $\pi$. For any initial distribution $\mu$, we have
\begin{align}
|w_{\beta}(\pi,\mu)-w_{\beta}(\pi^{k},\mu)|\leq K (1/k)^{1/d},
\label{eq12}
\end{align}
where $K=\frac{\alpha}{1-\beta}(K_{1}-\beta K_{2}M+\frac{2\beta MK_{2}}{1-\beta})$ with $M\coloneqq\sup_{(x,a)\in\sX\times\sA}|c(x,a)|$ under assumptions (h), (j) and (k).
\label{thm6}
\end{theorem}

\subsection{Expected Average Cost Case}
\label{sub2sec3}

In this section, as in Section~\ref{sub2sec2} we approach the problem by writing the expected average cost as an integral of the one stage cost function with respect to an invariant probability measure for the induced stochastic kernel. This way we obtain a bound on the difference between the actual and the approximated costs. However, the bound for this case will depend both on the rate of the quantizer approximating the actual policy and an extra term which changes with the system parameters. We will show that this extra term goes to zero as $n\rightarrow\infty$.

Note that for any $\pi \in S$, induced by $f$, assumption (l) implies that $\nu_{\pi}$ is an unique invariant probability measure for $Q_{\pi}$ and that $w_A(\pi,x)=\nu_{\pi}(c_{f})$ for all $x$ when $c$ is as in the assumption (c). The following theorem basically follows from Proposition~\ref{prop8} and the assumption (l). It is proved in Appendix~\ref{app5}.

\begin{theorem}
Let $\pi\in S$ and $\{\pi^{k}\}$ be the quantized approximations of $\pi$. Under assumptions (h), (j), (k),
and (l), for any $x\in\sX$ we have
\begin{align}
|w_A(\pi,x)-w_A(\pi^{k},x)|\leq 2MC \kappa^{n}+K_n(1/k)^{1/d}
\label{eq16}
\end{align}
for all $n\geq 0$, where $K_n=\bigl((2n-1)K_{2}\alpha M+K_{1}\alpha\bigr)$ and $M\coloneqq \sup_{(x,a)\in\sX\times\sA}|c(x,a)|$.
\label{thm7}
\end{theorem}

Observe that depending on the values of $C$ and $\kappa$, we can first make the first term in (\ref{eq16}) small enough by choosing sufficiently large $n$, and then for this $n$ we can choose $k$ large enough such that the second term  in (\ref{eq16}) is small.

\emph{Order Optimality:} The following example demonstrates that the order of approximation errors  in Theorems~\ref{thm7} and \ref{thm6} cannot be better than $O((\frac{1}{k})^{\frac{1}{d}})$. More precisely, we exhibit a simple standard example where we can lower bound the approximation errors for the optimal stationary policy by $L(1/k)^{1/d}$, for some positive constant $L$.

In what follows $h(\,\cdot\,)$ and $h(\,\cdot\,|\,\cdot\,)$ denote differential and conditional differential entropies, respectively \cite{CoTh06}.

\begin{example}\label{exm4}
Consider the linear system
\begin{align}
x_{n+1} = Ax_n + Ba_n + v_n, n=0,1,2,\ldots, \nonumber
\end{align}
where $\sX = \sA = \R^d$ and the $v_n$'s are i.i.d. random vectors whose common distribution has density $g$. For simplicity suppose that the initial distribution $\mu$ has the same density $g$. It is assumed that the differential entropy $h(g) \coloneqq -\int_{\sX} g(x) \log{g(x)} dx$ is finite. Let the one stage cost function be $c(x,a) \coloneqq \| x-a \|$. Clearly, the optimal stationary policy $\pi^{*}$ is induced by the identity $f(x) = x$, having the optimal cost $w_i(\pi,\mu) = 0$, where $i \in \{\beta,A\}$. Let $\{\pi^k\}$ be the quantized approximations of $\pi^{*}$. Fix any $k$ and define $D_n \coloneqq E_{\mu}^{\pi^k}\bigl[c(x_n,a_n)\bigr]$ for all $n$. Then, by the Shannon lower bound (SLB) \cite[p. 12]{YaTGr80} we have for $n\geq1$
\begin{align}
\log{k} \geq R(D_n) \geq h(x_n)+ \theta(D_n)
&= h(Ax_{n-1}+Ba_{n-1}+v_{n-1}) + \theta(D_n) \nonumber \\
&\geq h(Ax_{n-1}+Ba_{n-1}+v_{n-1}|x_{n-1},a_{n-1}) + \theta(D_n) \nonumber \\
&= h(v_{n-1}) + \theta(D_n), \label{eq25}
\end{align}
where $\theta(D_n) = - d + \log\biggl(\frac{1}{dV_d\Gamma(d)}\bigl(\frac{d}{D_n}\bigr)^d\biggr)$, $R(D_n)$ is the rate-distortion function of $x_n$, $V_d$ is the volume of the unit sphere $S_d = \{x: \|x\| \leq 1\}$, and $\Gamma$ is the gamma function. Here, (\ref{eq25}) follows from the independence of $v_{n-1}$ and the pair $(x_{n-1},a_{n-1})$. Note that $h(v_{n-1})=h(g)$ for all $n$. Hence, we obtain
$D_n \geq L (1/k)^{1/d}$, where $L \coloneqq \frac{d}{2} \bigl(\frac{2^{h(g)}}{d V_d \Gamma(d)}\bigr)^{1/d}$. This gives
$|w_{\beta}(\pi^{*},\mu) - w_{\beta}(\pi^k,\mu)| \geq \frac{L}{1-\beta} (1/k)^{1/d} \text{  and  }
|w_A(\pi^{*},\mu) - w_A(\pi^k,\mu)| \geq L (1/k)^{1/d}$.
\end{example}

\section{Approximation of Randomized Stationary Policies}
\label{sec4}

In this section, we extend results developed for the deterministic case
to randomized stationary policies. This extension is motivated by the facts that: (i) for a large
class of average cost optimization problems, it is not known whether one can
restrict the optimal policies to deterministic stationary policies, whereas
the optimality of possibly randomized stationary policies can be established
through the convex analytic method \cite{Man60,Bor02}, and (ii) randomized stationary policies
are necessary in constrained MDPs even for the discounted cost (see e.g. \cite{Piu97}).
Throughout this section we skip over all proofs
since these follow by applying same steps as in the proofs given in Section
\ref{sec2} and \ref{sec3}.

Throughout this section, we assume that conditions (a), (b), and (c) hold. Let $\pi\in
RS$ be induced by a stochastic kernel $\eta(da|x)$ on $\sA$ given $\sX$. By
Lemma 1.2 in \cite{GiSk79} there exists a measurable function
$\sf:\sX\times[0,1]\rightarrow\sA$ such that for any $E\in\B(\sA)$
\begin{align}
\eta(E|x)&=m\bigl(\{z: \sf(x,z)\in E\}\bigr), \nonumber
\end{align}
where $m$ is the Lebesgue measure on $[0,1]$. Equivalently, we can write $\eta(E|x)$ as
\begin{align}
\eta(E|x)&=\int_{[0,1]} \delta_{\sf(x,z)}(E) m(dz) . \label{neweq1}
\end{align}
Hence, $\pi$ can be represented as an (uncountable) convex combination of deterministic stationary policies parameterized by [0,1]. For each $z$, let $\{\sq_{k}(\,\cdot\,,z)\} \in \Q(\Lambda_k)$ denote the sequence of quantizers that uniformly converges to $\sf(\,\cdot\,,z)$ defined in Section~\ref{sec2}. Note that such quantizers can be constructed so that the resulting function $\sq_{k}(x,z)$ is measurable. Hence, $|\sq_k(\sX,z)|=|\Lambda_k|$ for all $z\in[0,1]$. Let $\{\pi^{k}\}$ be the sequence of randomized stationary policies induced by the stochastic kernels
\begin{align}
\eta_{k}(\,\cdot\,|x)\coloneqq\int_{[0,1]} \delta_{\sq_{k}(x,z)}(\,\cdot\,) m(dz).
\label{neweq2}
\end{align}

The following assumptions are versions of assumptions imposed in Sections~\ref{sec2} and \ref{sec3} adapted to randomized stationary policies. They will be imposed as needed throughout this section.
\begin{itemize}
\item [($\tilde{d}$)] $\sup_{\pi\in RS}\sum_{n=N+1}^{\infty}\int_{\sH_{\infty}} c_{n}(x_{n},a_{n}) P_{\mu}^{\pi}\rightarrow 0$ as $N\rightarrow\infty$.
\item [($\tilde{e}$)] For any $\pi\in RS$, $Q_{\pi}$ has a unique invariant probability measure $\nu_{\pi}$.
\item [($\tilde{f}1$)] The set $\Gamma_{RS}\coloneqq\{\nu\in\P(\sX): \nu Q_{\pi}=\nu \text{ for some } \pi\in RS\}$ is relatively sequentially compact in the setwise topology.
\item [($\tilde{f}2$)] There exists an $x\in\sX$ such that for all $B\in\B(\sX)$, $Q_{\pi}^{n}(B|x)\rightarrow\nu_{\pi}(B)$ uniformly in $\pi\in RS$.
\item [($\tilde{g}$)] $\sM\coloneqq\bigcap_{\pi\in RS}\sM_{\pi}\neq \emptyset$.
\item [($\tilde{l}$)] There exists a positive constant $C$ and $\kappa \in (0,1)$ such that for all $\pi \in RS$, there is a (necessarily unique) probability measure $\nu_{\pi} \in \P(\sX)$ satisfying
    \begin{align}
    \| \lambda^{\pi,x}_n - \nu_{\pi} \|_{TV} \leq C \kappa^n \text{  for all } x \in \sX \text{ and } n\geq1. \nonumber
    \end{align}
\end{itemize}

By adapting the proof of \cite[Lemma 3.3]{Her89} to randomized stationary policies, one can show that assumptions ($\tilde{e}$), ($\tilde{f2}$), and ($\tilde{g}$) are satisfied under any of the conditions $(i)$, $i\in\{1,2,\ldots,4\}$ in \cite[Section 3.3]{Her89}. Moreover, $\sM=\sX$ in (g) if at least one of the above conditions holds.
Furthermore, the statement in Proposition~\ref{fact3} remains true if we replace $S$ with $RS$. Hence, assumption (e1) implies ($\tilde{e}$), ($\tilde{f1}$) and ($\tilde{g}$) with $\sM = \sX$ in ($\tilde{g}$) if the invariant measures are unique.

\begin{example}
Let us again consider the additive-noise system of Example 1 with the same assumptions. Recall that boundedness of $F$ implies assumption (e1). On the other hand, it also implies condition $(2)$ in \cite[Section 3.3]{Her89}. Hence, if $F$ has a bounded range, then ($\tilde{e}$), ($\tilde{f}1$), ($\tilde{f}2$) and ($\tilde{g}$) with $\sM=\sX$ hold.
\end{example}

The first result in this section deals with problem \textbf{(P1)} for randomized policies. Recall that $w_t$, $w_{\beta}$ and $w_A$, respectively, denote the total, discounted, and average costs.

\begin{theorem}
Suppose assumptions (a), (b), (c) hold. Let $\pi \in RS$ and $\{\pi^k\}$ be the quantized approximations of $\pi$. Then,  $w_{\beta}(\pi^k,\mu) \rightarrow w_{\beta}(\pi,\mu)$ as $k\rightarrow\infty$. The same statement is true for $w_t$ if we further impose assumption ($\tilde{d}$).
Furthermore, if $\mu$ be concentrated on some $x \in \sM$, then $w_A(\pi^k,\mu) \rightarrow w_A(\pi,\mu)$ as $k\rightarrow\infty$ under the assumptions $(\tilde{e})$, $(\tilde{f}1)$ or $(\tilde{f}2)$, and $(\tilde{g})$.
\label{thm8}
\end{theorem}

The next result deals with problem \textbf{(P2)} in the randomized setting.

\begin{theorem}
Let $\pi\in RS$ and $\{\pi^{k}\}$ be the quantized approximations of $\pi$. Under assumptions (h), (j) and (k), for any initial distribution $\mu$ we have
\begin{align}
|w_{\beta}(\pi,\mu)-w_{\beta}(\pi^{k},\mu)|\leq (1/k)^{1/d} K, \nonumber
\end{align}
and on the other hand for all $x\in\sX$ and all $n\geq1$
\begin{align}
|w_A(\pi,x)-w_A(\pi^{k},x)|\leq 2MC\kappa^{n}+K_n(1/k)^{1/d} \nonumber
\end{align}
if we further assume ($\tilde{l}$). Here, $K$, $K_n$ ($n\geq1$) and $M$ are as in Theorems~\ref{thm6} and \ref{thm7}.
\label{thm10}
\end{theorem}



\section{Conclusion}\label{sec5}

In this paper, the problem of approximating deterministic stationary policies in MDPs was considered
for total, discounted, and average costs. We introduced deterministic stationary quantizer policies and
showed that any deterministic stationary policy can be approximated with an arbitrary precision by such policies. We also found upper bounds on the approximation errors in terms of the rates of the quantizers. These results were then extended to randomized stationary policies.

One direction for future work is to establish similar results for approximations
where the set of admissible quantizers has a certain structure, such as the set
of quantizers having convex codecells \cite{GyLi03}, which may give rise to
practical design methods. Moreover, if one can obtain further results on the structure of optimal
policies (e.g., by showing that an optimal policy satisfies a
Lipschitz property with a known bound on the constant),
the results in this paper may be directly applied to obtain approximation
bounds for quantized policies.  As a final remark, since setwise continuity assumption might be too restrictive in certain important cases, it is of interest to study a version of this problem where the setwise continuity assumption is replaced with the weak continuity in the state-action variables.

\section*{APPENDIX} \label{sec5}

\subsection{\textbf{Proof of Proposition \ref{prop1}}}
\label{app1}

We need to prove that $P_{\mu}^{\pi^{k}}(g) \rightarrow P_{\mu}^{\pi}(g)$ for any $g\in\bigcup_{n=0}^{\infty}\C(\sH_{n})$. Suppose $g\in\C(\sH_{n})$ for some $n$. Then we have $P_{\mu}^{\pi^{k}}(g) = \lambda^{\pi^k,\mu}_{(n)}(g_{q_k})$ and $P_{\mu}^{\pi}(g) = \lambda^{\pi,\mu}_{(n)}(g_{f})$. Note that both $g_f$ and $g_{q_k}$ ($k\geq1$) are uniformly bounded. Since $g$ is continuous in the $``a"$ terms by definition and $q_k$ converges to $f$, we have $g_{q_k} \rightarrow g_f$. Hence, by Theorem~\ref{cor1} it is enough to prove that $\lambda^{\pi^k,\mu}_{(n)}\rightarrow \lambda^{\pi,\mu}_{(n)}$ setwise as $k\rightarrow\infty$.

We will prove this by induction. Clearly,
$\lambda^{\pi^k,\mu}_{(1)}\rightarrow \lambda^{\pi,\mu}_{(1)}$
setwise by assumption (a). Assume the claim is true for some $n\geq 1$. For any $h\in B(\sX^{n+2})$ we can write
$\lambda^{\pi^k,\mu}_{(n+1)}(h)=\lambda^{\pi^k,\mu}_{(n)}\bigl(\lambda^{\pi^k,x_n}_{(1)}(h)\bigr)$ and
$\lambda^{\pi,\mu}_{(n+1)}(h)=\lambda^{\pi,\mu}_{(n)}\bigl(\lambda^{\pi,x_n}_{(1)}(h)\bigr)$.
Since $\lambda^{\pi^k,x_n}_{(1)}(h)\rightarrow \lambda^{\pi,x_n}_{(1)}(h)$ for all $(x_{0},\ldots,x_{n})\in\sX^{n+1}$ by assumption (a) and $\lambda^{\pi^k,\mu}_{(n)}\rightarrow\lambda^{\pi,\mu}_{(n)}$ setwise, we have $\lambda^{\pi^k,\mu}_{(n+1)}(h) \rightarrow \lambda^{\pi,\mu}_{(n+1)}(h)$ by by Theorem \ref{cor1} which completes the proof.

\subsection{\textbf{Proof of Theorem \ref{thm4}}}
\label{app2}

Let $Q_{\pi}$ and $Q_{\pi^{k}}$ be the stochastic kernels, respectively, for $\pi$
and $\{\pi^k\}$ defined in (\ref{eq4}). By assumption (e), $Q_{\pi}$ and $Q_{\pi^{k}}$ ($k\geq1$) have
unique, and so ergodic, invariant probability measures $\nu_{\pi}$ and $\nu_{\pi^{k}}$,
respectively. Since $x\in\sM$, we have $w_A(\pi^{k},\mu)=\nu_{\pi^{k}}(c_{q_k})$ and $w_A(\pi,\mu)=\nu_{\pi}(c_{f})$. Observe that
$c_{q_k}(x) \rightarrow c_{f}(x)$ for all
$x$ by assumption (c). Hence, if we prove $\nu_{\pi^k} \rightarrow \nu_{\pi}$ setwise, then
by Theorem \ref{cor1} we have $w_A(\pi^{k},\mu)\rightarrow w_A(\pi,\mu)$. We prove this first under (f1) and then under (f2).

\noindent \textit{I) Proof under assumption~(f1)}

We show that every setwise convergent subsequence $\{\nu_{\pi^{k_l}}\}$ of $\{\nu_{\pi^k}\}$ must
converge to $\nu_{\pi}$. Then, since $\Gamma_s$ is relatively sequentially compact in the setwise topology, there is
at least one setwise convergent subsequence $\{\nu_{\pi^{k_l}}\}$ of $\{\nu_{\pi^k}\}$, which implies the result.

Let $\nu_{\pi^{k_l}} \rightarrow \nu$ setwise for some $\nu \in \P(\sX)$. We will show that $\nu = \nu_{\pi}$ or equivalently
$\nu$ is an invariant probability measure of $Q_{\pi}$. For simplicity, we write  $\{\nu_{\pi^{l}}\}$ instead of $\{\nu_{\pi^{k_l}}\}$. Let $g\in B(\sX)$. Then by assumption (e) we have
\begin{align}
\nu_{\pi^{l}}(g)= \nu_{\pi^{l}}(Q_{\pi^{l}}g).\nonumber
\end{align}
Observe that by assumption (a), $Q_{\pi^{l}}g(x)\rightarrow Q_{\pi}g(x)$ for all $x$.
Since $Q_{\pi}g(x)$ and $Q_{\pi^{l}}g(x)$ ($l\geq1$) are uniformly
bounded and $\nu_{\pi^{l}}\rightarrow \nu$ setwise, we have $
\nu_{\pi^{l}}(Q_{\pi^{l}}g) \rightarrow \nu_{\pi}(Q_{\pi}g)$
by Theorem \ref{cor1}. On the other hand since $\nu_{\pi^{l}}\rightarrow \nu$ setwise we have
$\nu_{\pi^{l}}(g) \rightarrow \nu(g)$. Thus
$\nu(g)=\nu(Q_{\pi}g)$. Since $g$ is arbitrary, $\nu$ is an invariant probability measure for $Q_{\pi}$.

\noindent\textit{II) Proof under assumption~(f2)}

Observe that for all $x\in\sX$ and all $n$, $\lambda^{\pi^{k},x}_{n}\rightarrow \lambda^{\pi,x}_{n}$ setwise as $k\rightarrow\infty$ since $P_{x}^{\pi^{k}}\rightarrow P_{x}^{\pi}$ in the $ws^{\infty}$ topology (see Proposition \ref{prop1}). Let $B\in\B(\sX)$ be given and fix some $\varepsilon>0$. By assumption (f2) we can choose $N$ large enough such that $|\lambda^{\tilde{\pi},x}_{N}(B)-\nu_{\tilde{\pi}}(B)|<\varepsilon/3$ for all $\tilde{\pi}\in\{\pi,\pi^{1},\pi^{2},\cdots\}$. For this $N$, choose $K$ large enough such that $|\lambda^{\pi^{k},x}_{N}(B)-\lambda^{\pi,x}_{N}(B)|<\varepsilon/3$ for all $k\geq K$. Thus, for all $k\geq K$ we have
\begin{align}
  |\nu_{\pi^{k}}(B)-\nu_{\pi}(B)| \leq  |\nu_{\pi^{k}}(B)-\lambda^{\pi^{k},x}_{N}(B)|
  +|\lambda^{\pi^{k},x}_{N}(B)-\lambda^{\pi,x}_{N}(B)|
  +|\lambda^{\pi,x}_{N}(B)-\nu_{\pi}(B)|<\varepsilon.\nonumber
\end{align}
Since $\varepsilon$ is arbitrary, we obtain $\nu_{\pi^{k}}(B)\rightarrow\nu_{\pi}(B)$, which completes the proof.

\subsection{\textbf{Proof of Proposition \ref{prop8}}}
\label{app3}

We will prove this result by induction. Let $\mu$ be an arbitrary initial distribution and fix $k$. For $n=1$ the claim holds by the following argument:
\begin{align}
\|\lambda^{\pi,\mu}_1 - \lambda^{\pi^k,\mu}_1\|_{TV}
&= 2\sup_{B\in\B(\sX)} \bigl|\mu(\lambda^{\pi,x}_1(B)) - \mu(\lambda^{\pi^k,x}_1(B))\bigr|\nonumber \\
&\leq \mu\bigl(\|\lambda^{\pi,x}_1 - \lambda^{\pi^k,x}_1\|_{TV}\bigr) \nonumber \\
&\leq \mu\bigl(K_{2} d_{\sA}(f(x),q_{k}(x))\bigr) \text{ (by assumption (k))}\nonumber \\
&\leq\sup_{x\in\sX} K_{2} d_{\sA}(f(x),q_{k}(x))
\leq (1/k)^{1/d} K_{2} \alpha \text{ (by Lemma \ref{fact4}).}\nonumber
\end{align}
Observe that the bound $\alpha K_{2}(2n-1)(1/k)^{1/d}$ is independent of the choice of initial distribution $\mu$ for $n=1$. Assume the claim is true for $n\geq1$. Then we have
\begin{align}
  \|\lambda^{\pi,\mu}_{n+1} - \lambda^{\pi^k,\mu}_{n+1}\|_{TV}
  &=2\sup_{B\in\B(\sX)}\bigl|\lambda^{\pi,\mu}_1(\lambda^{\pi,x_1}_n(B)) -  \lambda^{\pi^k,\mu}_1(\lambda^{\pi^k,x_1}_n(B))\bigr| \nonumber \\
  &=2\sup_{B\in\B(\sX)}\bigl|\lambda^{\pi,\mu}_1(\lambda^{\pi,x_1}_n(B)) - \lambda^{\pi,\mu}_1(\lambda^{\pi^k,x_1}_n(B)) \nonumber \\
   &+ \lambda^{\pi,\mu}_1(\lambda^{\pi^k,x_1}_n(B)) - \lambda^{\pi^k,\mu}_1(\lambda^{\pi^k,x_1}_n(B))\bigr|\nonumber\\
  &\leq \lambda^{\pi,\mu}_1(\|\lambda^{\pi,x}_n - \lambda^{\pi^k,x}_n\|_{TV}) + 2 \|\lambda^{\pi,\mu}_1 - \lambda^{\pi^k,\mu}_1\|_{TV} \label{eq7}\\
  &\leq (1/k)^{1/d}(2n-1)K_{2}\alpha+2(1/k)^{1/d}K_{2}\alpha \label{eq8} \\
  &=\alpha K_{2}(2(n+1)-1)(1/k)^{1/d}\alpha. \nonumber
\end{align}
Here (\ref{eq7}) follows since
\begin{align}
|\mu(h) - \eta(h)|\leq \|\mu-\eta\|_{TV}\sup_{x\in\sX}|h(x)| \nonumber
\end{align}
and (\ref{eq8}) follows since the bound $\lambda K_{2}(2n-1)(1/k)^{1/d}$ is independent of the initial distribution.

\subsection{\textbf{Proof of Theorem \ref{thm6}}}
\label{app4}

For any fixed $k$ we have
\begin{align}
  |w_{\beta}(\pi)-w_{\beta}(\pi^{k})|
  &=\biggl|\sum_{n=0}^{\infty} \beta^{n} \lambda^{\pi,\mu}_n(c_f)
  -\sum_{n=0}^{\infty} \beta^{n} \lambda^{\pi^k,\mu}_n(c_{q_k})\biggr| \nonumber\\
  &\leq \sum_{n=0}^{\infty} \beta^{n}\bigl( \;| \lambda^{\pi,\mu}_n(c_f)-\lambda^{\pi,\mu}_n(c_{q_k})|
  +  | \lambda^{\pi,\mu}_n(c_{q_k})-\lambda^{\pi^k,\mu}_n(c_{q_k})| \; \bigr) \nonumber \\
  &\leq \sum_{n=0}^{\infty} \beta^{n} \bigl(\;
  \sup_{x_{n}\in\sX}|c_f-c_{q_k}|
  +\|\lambda^{\pi,\mu}_n-\lambda^{\pi,\mu}_n\|_{TV}M \; \bigr) \nonumber\\
  &\leq \sum_{n=0}^{\infty} \beta^{n} \biggl( \;\sup_{x_{n}\in\sX}
  d_{\sA}(f(x_n),q_k(x_n))K_{1}\biggr)  \nonumber \\
  &+\sum_{n=1}^{\infty} \beta^{n}
  \biggl( (1/k)^{1/d}(2n-1)K_{2}\alpha M\biggr)
\label{neweq16} \\
& \leq \sum_{n=0}^{\infty} \beta^{n} \biggl( (1/k)^{1/d}\alpha
K_{1}\biggr)
+\sum_{n=1}^{\infty} \beta^{n} \biggl( (1/k)^{1/d}(2n-1)K_{2}\alpha
M\biggr) \text{ (by Lemma \ref{fact4})} \nonumber \\
 &= (1/k)^{1/d} \alpha (K_{1}-\beta
 K_{2}M)\frac{1}{1-\beta}+(1/k)^{1/d}2K_{2}\alpha M\frac{\beta}{(1-\beta)^{2}}
 \nonumber \\
 &=(1/k)^{1/d}\frac{\alpha}{1-\beta}(K_{1}-\beta K_{2}M+\frac{2\beta MK_{2}}{1-\beta}).\nonumber
\end{align}
Here (\ref{neweq16}) follows from Assumption (j) and Proposition \ref{prop8}. This completes the proof.

\subsection{\textbf{Proof of Theorem \ref{thm7}}}
\label{app5}

For any $k$ and $x\in\sX$, we have
\begin{align}
|w_A(\pi,x)-w_A(\pi^{k},x)|
& = | \nu_{\pi}(c_f) - \nu_{\pi^{k}}(c_{q_k}) |  \nonumber \\
&\leq | \nu_{\pi}(c_f) - \nu_{\pi}(c_{q_k}) | + |\nu_{\pi}(c_{q_k}) - \nu_{\pi^k}(c_{q_k}) | \nonumber \\
&\leq \sup_{x\in\sX} |c_{f} - c_{q_k}|+\|\nu_{\pi}-\nu_{\pi^{k}}\|_{TV}\sup_{x\in\sX} |c_{q_k}|\nonumber \\
&\leq \sup_{x\in\sX} K_{1} d_{\sA}(f(x),q_{k}(x))+\|\nu_{\pi}-\nu_{\pi^{k}}\|_{TV} M \text{ (by assumption (j))}\nonumber \\
&\leq (1/k)^{1/d}K_{1}\alpha+\bigl(\|\nu_{\pi}-\lambda^{\pi,x}_{n}\|_{TV}+\|\lambda^{\pi,x}_{n}-\lambda^{\pi^k,x}_{n}\|_{TV}  +\|\lambda^{\pi^k,x}_{n}-\nu_{\pi^{k}}\|_{TV}\bigr)M  \nonumber \\
&\leq (1/k)^{1/d}K_{1}\alpha+\bigl(2C\kappa^{n}+(1/k)^{1/d}(2n-1)K_{2}\alpha\bigr)M \label{eq15}\\
&=2MC \kappa^{n}+\bigl((2n-1)K_{2}\alpha M+K_{1}\alpha\bigr)(1/k)^{1/d} \nonumber,
\end{align}
where (\ref{eq15}) follows from assumption (l) and Proposition \ref{prop8}.

\bibliographystyle{IEEEtran}

\begin{thebibliography}{10}
\providecommand{\url}[1]{#1}
\csname url@samestyle\endcsname
\providecommand{\newblock}{\relax}
\providecommand{\bibinfo}[2]{#2}
\providecommand{\BIBentrySTDinterwordspacing}{\spaceskip=0pt\relax}
\providecommand{\BIBentryALTinterwordstretchfactor}{4}
\providecommand{\BIBentryALTinterwordspacing}{\spaceskip=\fontdimen2\font plus
\BIBentryALTinterwordstretchfactor\fontdimen3\font minus
  \fontdimen4\font\relax}
\providecommand{\BIBforeignlanguage}[2]{{%
\expandafter\ifx\csname l@#1\endcsname\relax
\typeout{** WARNING: IEEEtran.bst: No hyphenation pattern has been}%
\typeout{** loaded for the language `#1'. Using the pattern for}%
\typeout{** the default language instead.}%
\else
\language=\csname l@#1\endcsname
\fi
#2}}
\providecommand{\BIBdecl}{\relax}
\BIBdecl







\bibitem{Bor02}
V.~Borkar, ``Convex analytic methods in {M}arkov decision processes,'' in
  \emph{Handbook of Markov Decision Processes}, E.~Feinberg and A.~Shwartz,
  Eds.\hskip 1em plus 0.5em minus 0.4em\relax Kluwer Academic Publisher, 2002.


\bibitem{MeTw94}
S.~P. Meyn, R.~L. Tweedie, \emph{Markov chains and stochastic stability}.\hskip 1em plus
  0.5em minus 0.4em\relax Springer-Verlag, 1994.



\bibitem{CTCN}
S.~P. Meyn, \emph{Control Techniques for Complex Networks}.\hskip 1em plus
  0.5em minus 0.4em\relax Cambridge University Press, 2007.



\bibitem{BeTs96}
D.~Bertsekas and J.~Tsitsiklis, \emph{Neuro-Dynammic Programming}.\hskip 1em
  plus 0.5em minus 0.4em\relax Athena Scientific, 1996.


\bibitem{ReKr02}
Z.~Ren and B.~Krogh, ``State aggregation in {M}arkov decision processes,'' in
  \emph{CDC 2002}, Las Vegas, December 2002.


\bibitem{Ort07}
R.~Ortner, ``Pseudometrics for state aggregation in average reward {M}arkov
  decision processes,'' in \emph{Algorithmic Learning Theory}.\hskip 1em plus
  0.5em minus 0.4em\relax Springer-Verlag, 2007.



\bibitem{Whi82}
D.~White, ``Finite-state approximations for denumerable state infinite horizon
  discounted {M}arkov decision processes with unbounded rewards,'' \emph{J.
  Math. Anal. Appl.}, vol. 186, pp. 292--306, 1982.

\bibitem{Cav86}
R.~Cavazos-Cadena, ``Finite-state approximations for denumerable state
  discounted {M}arkov decision processes,'' \emph{Appl. Math. Optim.}, vol.~14,
  pp. 1--26, 1986.




\bibitem{Man60}
A.~Manne, ``Linear programming and sequential decisions,'' \emph{Management
  Sciences}, vol.~6, no.~3, pp. 259--267, 1980.

\bibitem{Fei96}
E.~Feinberg, ``On measurability and representation of strategic measures in
  {M}arkov decision processes,'' \emph{Statistics, Probability and Game
  Theory}, vol.~30, pp. 29--43, 1996.

\bibitem{Bil99}
P.~Billingsley, \emph{Convergence of probability measures}, 2nd~ed.\hskip 1em
  plus 0.5em minus 0.4em\relax New York: Wiley, 1999.


\bibitem{HeLa96}
O.~Hern\'andez-Lerma and J.~Lasserre, \emph{Discrete-Time {M}arkov Control
  Processes: Basic Optimality Criteria}.\hskip 1em plus 0.5em minus 0.4em\relax
  Springer, 1996.


\bibitem{HeLa99}
O.~Hern\'andez-Lerma and J.~Lasserre, \emph{Further Topics on Discrete-time Markov Control Processes}.\hskip 1em plus 0.5em minus 0.4em\relax
  Springer, 1999.


\bibitem{HeLa03}
O.~Hern\'andez-Lerma and J.~Lasserre, \emph{{M}arkov Chains and Invariant
  Probabilities}.\hskip 1em plus 0.5em minus 0.4em\relax Birkhauser, 2003.

\bibitem{Sch75}
M.~Sch\"{a}l, ``On dynamic programming: compactness of the space of policies,''
  \emph{Stochastic Process. Appl.}, vol.~3, no.~4, pp. 345--364, 1975.

\bibitem{Bal89}
E.~Balder, ``On the compactness of the space of policies in stochastic dynamic
  programming,'' \emph{Stochastic Process. Appl.}, vol.~32, no.~1, pp.
  141--150, 1989.

\bibitem{Now88}
A.~Nowak, ``On the weak topology on a space of probability measures induced by
  policies,'' \emph{Bull. Polish Acad. Sci. Math.}, vol.~36, pp. 181--186,
  1988.

\bibitem{Ser82}
R.~Serfozo, ``Convergence of {L}ebesgue integrals with varying measures,''
  \emph{Sankhya Ser.A}, pp. 380--402, 1982.

\bibitem{HeMoRo91}
O.~Hern\'andez-Lerma, R.~Montes-De-Oca, and R.~Cavazos-Cadena, ``Recurrence
  conditions for {M}arkov decision processes with {B}orel state space: a
  survey,'' \emph{Ann. Oper. Res.}, vol.~28, no.~1, pp. 29--46, 1991.

\bibitem{Hai06}
M.~Hairer, ``Ergodic properties of {M}arkov processes,'' \emph{Lecture Notes},
  2006.

\bibitem{HeRo01}
O.~Hern\'andez-Lerma and R.~Romera, ``Limiting discounted-cost control of
  partially observable stochastic systems,'' \emph{SIAM J. Control Optim.},
  vol.~40, no.~2, pp. 348--369, 2001.

\bibitem{HeLa94}
O.~Hern\'andez-Lerma and J.B.~Lasserre, ``Linear programming and average optimality of Markov control
processes on Borel spaces - unbounded costs,'' \emph{SIAM J. Control Optim.},
  vol.~32, no.~2, pp. 480--500, 1994.

\bibitem{Her89}
O.~Hern\'andez-Lerma, \emph{Adaptive {M}arkov Control Processes}.\hskip 1em
  plus 0.5em minus 0.4em\relax Springer-Verlag, 1989.


\bibitem{GiSk79}
I.~Gihman and A.~Skorohod, \emph{Controlled Stochastic Processes}.\hskip 1em
  plus 0.5em minus 0.4em\relax Springer-Verlag, 1979.

\bibitem{GyLi03}
A.~Gy\"orgy and T.~Linder, ``Codecell convexity in optimal entropy-constrained
  vector quantization,'' \emph{IEEE Trans. Inf. Theory}, vol.~49, no.~7, pp.
  1821--1828, July 2003.


\bibitem{FeKaZa12}
E.A.~Feinberg and P.O.~Kasyanov and N.V.~Zadioanchuk, ``Average cost {M}arkov decision processes with weakly continuous transition probabilities,'' \emph{Math. Oper. Res.}, vol.~37, no.~4, pp. 591--607, Nov., 2012.

\bibitem{Piu97}
A.B.~Piunovskiy, \emph{Optimal Control of Random Sequences in Problems with Constraints}.\hskip 1em
  plus 0.5em minus 0.4em\relax Kluwer, 1997.

\bibitem{FeKaZa13}
E.A.~Feinberg and P.O.~Kasyanov and N.V.~Zadioanchuk, ``Berge's theorem for noncompact image sets,'' \emph{J.
Math. Anal. Appl.}, vol.~397, pp. 255--259, 2013.


\bibitem{DuPr12}
F.~Dufour and T.~Prieto-Rumeau, ``Finite linear programming approximations of constrained discounted Markov decision processes,'' \emph{SIAM J. Control Optim.}, vol.~51, no.~2, pp. 1298--1324, 2013.



%
%


\bibitem{ChBo06}
C.D.~Aliprantis and K.C.~Border, \emph{Infinite Dimensional Analysis: A Hitchhiker's Guide}.\hskip 1em
  plus 0.5em minus 0.4em\relax Springer, 2006.


\bibitem{CoTh06}
T.M.~Cover and J.A.~Thomas, \emph{Elements of Information Theory}.\hskip 1em
  plus 0.5em minus 0.4em\relax Wiley, 2006.


\bibitem{GrNe98}
R.M.~Gray and D.L.~Neuhoff, ``Quantization,''  \emph{IEEE Trans. Inf. Theory}, vol.~44, no.~6, pp.2325--2383, Oct. 1998.

\bibitem{YaTGr80}
Y.~Yamada, S.~Tazaki and R.M.~Gray, ``Asymptotic performance of block quantizers with difference distortion measures,''  \emph{IEEE Trans. Inf. Theory}, vol.~26, pp.6--14, Jan. 1980.

\end{thebibliography}


\end{document}